\documentclass{amsart}

\usepackage{latexsym,amsmath,amsfonts,amscd,amssymb, amsthm}

\theoremstyle{plain}
\newtheorem{theorem}{Theorem}[section]
\newtheorem{lemma}{Lemma}[section]
\newtheorem{proposition}{Proposition}[section]
\newtheorem{corollary}{Corollary}[section]

\theoremstyle{definition}
\newtheorem{definition}{Definition}[section]
\newtheorem{example}{Example}[section]

\theoremstyle{remark}
\newtheorem{remark}{Remark}[section]

\title{Symmetries and regularity for holomorphic maps between balls}

\author{John P. D'Angelo}

\address{Dept. of Mathematics, Univ. of Illinois, 1409 W. Green St., Urbana IL 61801}

\email{jpda@math.uiuc.edu}

\author{Ming Xiao}

\address{Dept. of Mathematics, Univ. of Illinois, 1409 W. Green St., Urbana IL 61801}

\curraddr{ Dept. of Mathematics, Univ. of California San Diego, 9500 Gilman Dr.,  La
Jolla CA 92093-0112}

\email{m3xiao@ucsd.edu}

\begin{document}

\maketitle

\begin{abstract} Let $f:{\mathbb B}^n \to {\mathbb B}^N$ be a holomorphic map.
We study subgroups $\Gamma_f \subseteq {\rm Aut}({\mathbb B}^n)$ and $T_f \subseteq {\rm Aut}({\mathbb B}^N)$.
When $f$ is proper, we show both these groups are Lie subgroups.
When $\Gamma_f$ contains the center
of ${\bf U}(n)$, we show that $f$ is spherically equivalent to a polynomial. When $f$ is minimal we show
that there is a homomorphism $\Phi:\Gamma_f \to T_f$ such that $f$ is equivariant with respect to $\Phi$.
To do so, we characterize minimality via the triviality 
of a third group $H_f$. We relate properties
of ${\rm Ker}(\Phi)$ to older results on invariant proper maps between balls. When $f$ is proper but completely non-rational,
we show that either both $\Gamma_f$ and $T_f$ are finite or both are noncompact.

\medskip

\noindent
{\bf AMS Classification Numbers}: 32H35, 32H02, 32M05, 32A50, 22E99, 51F25.

\medskip

\end{abstract}

\section{Introduction}

Let ${\mathbb B}^n$ denote the unit ball in complex Euclidean space ${\mathbb C}^n$.
We consider bounded holomorphic maps $f:{\mathbb B}^n \to {\mathbb C}^N$; after 
division by a constant we assume that the image of the ball under $f$ lies 
in the unit ball ${\mathbb B}^N$. Following
[DX] we study various groups associated with such an $f$. Our set-up and many
of the results require neither 
that the map be proper nor that it be rational. Under these assumptions, however, we 
obtain additional information.

The holomorphic automorphism group ${\rm Aut}({\mathbb B}^n)$ is transitive.
It is a Lie group that can be regarded as the quotient of ${\bf SU}(n,1)$ by its center
or as a collection of linear fractional transformations.
Holomorphic maps $f,g: {\mathbb B}^n \to {\mathbb B}^N$ are {\bf spherically equivalent}
if there are automorphisms $\gamma$ and $\psi$ such that $\psi \circ g = f \circ \gamma$.

In [DX] the authors associated 
a subgroup $A_f$ of ${\rm Aut}({\mathbb B}^n) \times {\rm Aut}({\mathbb B}^N)$ to a proper map $f$.
This group $A_f$ is 
defined to be those pairs $(\gamma, \psi)$ for which $f \circ \gamma = \psi \circ f$.
Note that this definition makes sense whenever the image of $f$ is contained in ${\mathbb B}^N$.
One of our key regularity results holds in this more general setting.

Most of the results in [DX] involve properties of $\Gamma_f$, the projection
of $A_f$ onto its first factor. In this paper we develop further uses of this group,
including results that make no regularity assumptions about $f$. In Proposition 2.3 
we show that $\Gamma_f$ is a Lie subgroup of ${\rm Aut}({\mathbb B}^n)$, without assuming $f$ is rational.
We also use properties of $T_f$, the projection
onto the second factor. Here properness is used to show that $T_f$ is a (closed) Lie subgroup.

A map $f:{\mathbb B}^n \to {\mathbb C}^N$ is called {\bf minimal} if its image lies in no lower dimensional affine linear subspace.
See the survey article [HJY]. 
Let $H_f$ denote the subgroup of $T_f$ consisting of those target automorphisms $\psi$ for which $\psi \circ f = f$. 
In Proposition 3.1 we show that $H_f$ is trivial if and only if $f$ is minimal. 
When $f:{\mathbb B}^n \to {\mathbb B}^N$ is minimal, we prove (Theorem 3.1) 
that there is a group homomorphism $\Phi: \Gamma_f \to T_f$, and hence $f$ is equivariant with respect to $\Phi$. 
When $f:{\mathbb B}^n \to {\mathbb B}^N$ is also proper, the kernel of $\Phi$
depends on the boundary regularity of $f$. The kernel of $\Phi$ coincides with the group $G_f$, as defined in [DX].
Thus $G_f$ is the subgroup of $\Gamma_f$ consisting of those source automorphisms $\gamma$ for which $f \circ \gamma = f$.

We next summarize the results in the paper. We start with the following result.

\begin{corollary} Suppose $f:{\mathbb B}^n \to {\mathbb B}^N$ is holomorphic.
If $\Gamma_f$ contains the center of ${\bf U}(n)$, then
$f$ is spherically equivalent to a polynomial. \end{corollary}

This corollary follows from the next theorem, where we give a general condition
implying that a holomorphic map $f$ is a polynomial.
Since $\Gamma_f$ is a Lie group, $\Gamma_f \cap {\bf U}(n)$ also is. The Lie algebra $\mathfrak{g}$
of $\Gamma_f \cap {\bf U}(n)$ consists of skew-Hermitian matrices $M$. For $M \in \mathfrak{g}$ put $L=-iM$.
Then $L$ has real eigenvalues. We prove Theorem 1.1 in Section 4 and obtain the consequences in Corollary 1.2.

\begin{theorem} Let $f: {\mathbb B}^n \to {\mathbb B}^N$ be holomorphic with $f(0)=0$.
If the Lie algebra $\mathfrak{g}$ of $\Gamma_f \cap {\bf U}(n)$
contains a matrix $M$ such that $L=-iM$ has 
$k$ positive eigenvalues, then there is a $k$-dimensional linear subspace $V$
such that the restriction of $f$ to $V$ is a polynomial.  \end{theorem}

\begin{corollary} Let $f: {\mathbb B}^n \to {\mathbb B}^N$ be a proper holomorphic mapping.
\begin{itemize}
\item  Assume
$\Gamma_f$ contains the maximal $n$-torus in ${\rm Aut}({\mathbb B}^n)$. Then $f$ is spherically equivalent to a monomial map.
\item Assume $\Gamma_f$ contains ${\bf U}(n)$. Then $f$ is spherically equivalent to an orthogonal sum of tensor products. 

\item Suppose $\Gamma_f$ contains the center of ${\bf U}(n)$ and is noncompact.
Then $f$ is a linear fractional transformation.

\end{itemize}
\end{corollary}

These statements appear in [DX] when $f$ is assumed to be rational. Recall from [F1], when $n\ge 2$,
that a proper holomorphic map between balls, with sufficient boundary regularity, must be rational.
Here we show that Hermitian invariance under a large
group forces rationality. In particular, if the group $\Gamma_f$ contains the maximal $n$-torus in ${\rm Aut}({\mathbb B}^n)$, 
then $f$ is rational.  Once we know that the map is rational, the results from [DX]
yield the additional information about spherical equivalence in this corollary.
At the end of this paper we mention a rigidity problem discussed in [S] and [CM]
and related to the third conclusion of the corollary.

The projections of $A_f$ onto each of the factors will be significant in this paper.
It is natural to ask when there is a
homomorphism from $\Gamma_f$ to $T_f$. We answer this question (Theorem 3.1) as follows.

\begin{theorem} Let $f: \mathbb{B}^n \to \mathbb{B}^N$ be a minimal proper map. Then there is
a homomorphism $\Phi: \Gamma_f \to T_f \subseteq \mathrm{Aut}(\mathbb{B}^N)$. 
Thus $f$ is equivariant with respect to $\Phi$.
\end{theorem}

Proposition 3.1 shows that $f$ is minimal if and only if $H_f$ is trivial.
Another name for minimal is {\it target essential}. The related notion
of {\it source essential} appears in [DX] and is expressed there in terms of 
$\Gamma_f$.

Section 3 also gathers some older known results (see for example [D], [F2], [Li1], [Li2]) about the groups $G_f$
and places them in the context of this paper. As noted above, for minimal proper maps, 
$G_f$ is the kernel of $\Phi$. The possible invariant groups $G_f$ for 
proper maps depend upon the regularity assumptions. As has been long known, the
invariant group for a rational proper map must be cyclic ([Li2]) and the list of possible representations
is short ([D]). On the other hand, every finite group arises when we drop
the assumption of rationality. 

Section 4 provides the proof of Theorem 4.1, the first theorem stated above. 

Section 5 considers {\bf completely non-rational} proper maps $f$; that is,
for each positive-dimensional affine linear subspace $V$ of ${\mathbb C}^n$, the restriction of $f$ to $V$
is not rational. We prove the following result.

\begin{theorem} Let $f$ be a minimal completely non-rational proper mapping between balls.
Then $\Gamma_f$ and $T_f$ are either both finite or both noncompact. \end{theorem}

The first author acknowledges support from NSF Grant DMS 13-61001. Both authors thank the referee for suggesting
several additional references. 

\section{preliminaries}

The inner product of $z,w$ in complex Euclidean space is denoted by $\langle z,w \rangle$
and the Euclidean squared norm by $||z||^2$. We use the same notation in each dimension.
Let ${\mathbb B}^n$ denote the unit ball in ${\mathbb C}^n$ and let ${\bf U}(n)$ denote the unitary group.
As noted in the first paragraph of the introduction, $f$ is not assumed to be proper in the following crucial definitions.

\begin{definition} Suppose $f:{\mathbb B}^n \to {\mathbb B}^N$ is holomorphic.
Then $A_f$ is the subgroup of ${\rm Aut}({\mathbb B}^n) \times {\rm Aut}({\mathbb B}^n)$
consisting of those pairs $(\gamma, \psi)$ for which
 $$ f \circ \gamma = \psi \circ f. $$ 
\end{definition}

\begin{definition} In the setting of Definition 2.1, we consider the following groups:
\begin{itemize}
\item $\Gamma_f$ is the projection of $A_f$ on the first factor.
\item $T_f$ is the projection of $A_f$ on the second factor.
\item $G_f = \{ \gamma \in {\rm Aut}({\mathbb B}^n) : f \circ \gamma =f \}$.
\item $H_f = \{ \psi \in {\rm Aut}({\mathbb B}^N) : \psi \circ f =f \}$.
\end{itemize} \end{definition}

We first observe the following simple facts about spherical equivalence.

\begin{proposition}
Let $f, g: \mathbb{B}^n \to \mathbb{B}^N$ be spherically equivalent.
Put $g=\alpha \circ f \circ \beta$. Then $H_g=\alpha \circ H_f \circ \alpha^{-1}.$
\end{proposition}
\begin{proof} Suppose $\psi_f \circ f = f$. Then
$$ \alpha \circ \psi_f \circ \alpha^{-1} \circ g = \alpha \circ \psi_f \circ \alpha^{-1} \circ \alpha \circ f \circ \beta = \alpha \circ \psi_f \circ f \circ \beta = \alpha \circ f \circ \beta = g. $$ \end{proof}

\begin{proposition}
Let $f, g :\mathbb{B}^n \to \mathbb{B}^N$ be spherically equivalent.
Put $g=\alpha \circ f \circ \beta$.  Then 
$$A_g=\{(\beta^{-1} \circ \gamma \circ \beta, \alpha \circ \psi \circ \alpha^{-1}):
(\gamma, \psi) \in A_f \}.$$
\end{proposition}
\begin{proof} The proof is a formal calculation similar to the previous proof. \end{proof}

\begin{proposition} Let $f:\mathbb{B}^n \to \mathbb{B}^N$ be a proper holomorphic map. 
Then 
\begin{itemize}
\item   $\Gamma_f$ is a Lie subgroup of ${\rm Aut}({\mathbb B}^n)$.
\item $T_f$ is a Lie subgroup of ${\rm Aut}({\mathbb B}^N)$.
\item $\Gamma_f$ is noncompact if and only if $T_f$ is noncompact.
\end{itemize}
\end{proposition}

\begin{proof} Both $\Gamma_f$ and $T_f$ are subgroups; we must therefore show that they are closed.
That $\Gamma_f$ is closed relies only on $f$ being continuous; that $T_f$ is closed relies
on $f$ being continuous and proper.

We first note that $\gamma_\nu \to \gamma$ in ${\rm Aut}({\mathbb B}^k)$
if $\gamma_\nu(z) \to \gamma(z)$ for each $z \in {\mathbb B}^k$. Here $k$ can be $n$ or $N$.
Recall in each case that an automorphism has the form $L \phi_c$ where
$L \in {\bf U}(k)$ and $\phi_c$ is a linear fractional transformation with $\phi_c(0)=c$.

After composing with an automorphism we assume $f(0)=0$.
Consider a sequence $\gamma_\nu \in \Gamma_f$
converging to $\gamma$. By definition there is a sequence $\psi_\nu \in {\rm Aut}({\mathbb B}^N)$
such that 
$$ f \circ \gamma_\nu = \psi_\nu \circ f. $$
Write $\gamma_\nu = U_\nu \phi_{a_v}$ and $\psi_\nu = V_\nu \varphi_{b_v}$ as above.  Evaluating at $0$ gives
$$ f(U_\nu a_\nu) = V_\nu b_\nu. \eqno (1) $$
If $\gamma_v$ converges to $\gamma \in {\rm Aut}({\mathbb B}^n)$, then the sequence $\{a_\nu\}$
lies in a compact subset of ${\mathbb B}^n$. Since $U_\nu$ is unitary and $f$ is continuous,
the left-hand side of (1) lies in a compact subset of ${\mathbb B}^N$. Therefore, since $V_\nu$ is unitary,
the sequence $\{b_\nu\}$ lies in a compact subset of ${\mathbb B}^N$ as well. 

Next consider a sequence $\psi_v \in T_f$ converging in ${\rm Aut}({\mathbb B}^N)$. 
Using the above notation, now $\{b_\nu\}$, and hence $V_\nu b_\nu$, lies in
a compact subset of the target ball ${\mathbb B}^N$. Since $f$ is continuous and proper, the inverse image of a compact
set is compact. Hence (1) implies that $\{U_\nu a_\nu\}$, and hence $\{a_\nu\}$, lies in a compact subset of ${\mathbb B}^n$.

Thus, if $\gamma_\nu \in \Gamma_f$ and $\gamma_\nu$ converges to $\gamma$, or if $\psi_\nu \in T_f$
and $\psi_\nu$ converges to $\psi$, then the sequences $\{a_\nu\}$ or $\{b_\nu\}$ are bounded
and hence have convergent subsequences. The unitary groups are compact,
and hence the $U_\nu$ and the $V_\nu$ also have convergent subsequences.
Thus there is a subsequence of source automorphisms
converging to a source automorphism $\gamma$ and a corresponding target automorphism $\psi$ with $f \circ \gamma = \psi \circ f$.
Thus both $\Gamma_f$ and $T_f$ are closed subgroups of their respective automorphism groups.
Hence each is a Lie subgroup.

The third statement has a similar proof. If $\Gamma_f$ contains a sequence of automorphisms $U_\nu \phi_{a_\nu}$
where $\{a_\nu\}$ tends to the boundary sphere in the source, then the corresponding sequence $V_\nu \varphi_{b_\nu}$, by properness, must have $\{b_\nu\}$ tending to the sphere in the target, and conversely.

\end{proof}

An automorphism of the ball preserves the origin if and only if it is unitary. This fact 
leads to the following lemma.

\begin{lemma} Suppose $f:{\mathbb B}^n \to {\mathbb B}^N$ is holomorphic and $f(0)=0$.
Let $\gamma \in {\bf U}(n)$. Then $\gamma \in \Gamma_f$ if and only if
$||f \circ \gamma||^2 = ||f||^2$.
\end{lemma}

\begin{proof} Suppose first that $||f \circ \gamma||^2 = ||f||^2$.
Then, by a well-known result from [D], there is a $U \in {\bf U}(N)$ such that
$f \circ \gamma = U \circ f$. Thus $\gamma \in \Gamma_f$.
Conversely, suppose $\gamma \in \Gamma_f$. Then  $f \circ \gamma = \psi_\gamma \circ f$ for some target
automorphism $\psi_\gamma$.
Evaluating at $0$ and using $f(0)=0$ gives $\psi_\gamma(0)=0$. Hence $\psi_\gamma$
is unitary and 
$$||f \circ \gamma||^2 = ||\psi_\gamma \circ f||^2 = ||f||^2 . $$
\end{proof}

\section{Equivariance}

The primary purpose of this section is to prove the following theorem.

\begin{theorem}
Let $f: \mathbb{B}^n \to \mathbb{B}^N$ be a minimal holomorphic map. Then there is
a homomorphism $\Phi: \Gamma_f \to T_f \subseteq \mathrm{Aut}(\mathbb{B}^N)$. 
Thus $f$ is equivariant with respect to $\Phi$.
\end{theorem}

We also have the following characterization of minimality.

\begin{proposition} Let $f: \mathbb{B}^n \to \mathbb{B}^N$ be holomorphic.  
Then $H_f$ is the trivial group if and only if $f$ is minimal.
\end{proposition}
\begin{proof} Lemma 3.1 below states that $f$ minimal implies $H_f$ trivial.
Lemma 3.2 shows that $f$ not minimal implies $H_f$ not trivial.\end{proof}

\begin{lemma}
Let $f$ be a minimal map from $\mathbb{B}^n$ to $\mathbb{B}^N$. Then $H_f$ is
the trivial group consisting of the identity map ${\bf I}_N$.
\end{lemma}

\begin{proof}
By Proposition 2.1, the statement is true for $f$ if and only if it is true for
$\psi \circ f \circ \gamma$ with $\gamma \in \mathrm{Aut}(\mathbb{B}^n)$ and $\psi
\in \mathrm{Aut}(\mathbb{B}^N).$ Thus, without loss of generality, we can compose with a
target automorphism to make $f(0)=0$. Consider the span:
$$\mathrm{Span}_{\mathbb{C}}(f)=\{\sum_{i=1}^s \lambda_{i} f(z_i): z_i \in
\mathbb{B}^n \ {\rm and} \ \lambda_i \in \mathbb{C} \}. $$
Since $f$ is minimal, $\mathrm{Span}_{\mathbb{C}}\{ f\}= \mathbb{C}^N.$ Let $\psi \in H_f$. Then
$\psi \circ f=f$. We put $z=0$ to conclude $\psi(0)=0$. Thus $\psi \in {\bf U}(N)$.
Since  $\psi$ is linear and preserves every $f(z)$, it must preserve every element
in $\mathrm{Span}_{\mathbb{C}}(f)=\mathbb{C}^N.$ Thus $\psi$ is the identity
map. 
\end{proof}

\begin{lemma} Let $f$ be a minimal map from $\mathbb{B}^n$ to
$\mathbb{B}^m$. Let $g$ be a holomorphic map from $\mathbb{B}^n$ to $\mathbb{B}^N$ that is spherically
equivalent to $0 \oplus f$. Then $H_g$ is conjugate to ${\bf U}(k) \oplus {\bf I}_m$. 
\end{lemma}

\begin{proof}
It suffices to assume $g=0 \oplus f$. The general case follows from
this special case and Proposition 2.1. As usual, we can assume $f(0)=0$ and $g(0)=0$.
Since $f$ is minimal, we have $\mathrm{Span}_{\mathbb{C}} \{ g \}=0 \oplus
\mathbb{C}^m.$
Let $\psi \in H_f$. Then $\psi \circ f=f$. Again we let $z=0$ to get $\psi(0)=0.$ Thus
$\psi$ is unitary. As above, since $\psi$ is linear and preserves every $g(z)$,
 it preserves
$\mathrm{Span}_{\mathbb{C}}\{ g \}=0 \oplus \mathbb{C}^m$. Thus $\psi \in {\bf U}(k)
\oplus {\bf I}_m$. Hence $H_g \subseteq {\bf U}(k) \oplus {\bf I}_m$. The definition of $g$
yields the opposite inclusion. We conclude $H_g = {\bf U}(k) \oplus {\bf I}_m$.
\end{proof}

We are now ready to prove Theorem 3.1.

\bigskip

{\bf  Proof of Theorem 3.1:} We use minimality to show, for each $\gamma \in \Gamma_f$,
that there is a unique $\psi \in T_f$ such that $(\gamma, \psi) \in A_f$. Assume
$(\gamma, \psi_1), (\gamma, \psi_2) \in A_f$. Then 
$$(\gamma \circ \gamma^{-1},
\psi_1 \circ \psi_2^{-1})=({\bf I}_n, \psi_1 \circ \psi_2^{-1}) \in
A_f. $$ Thus $ \psi_1 \circ \psi_2^{-1} \in H_f.$ Proposition 3.1 implies that $H_f$ is
trivial and hence $\psi_1=\psi_2 $.

By the uniqueness of $\psi$, there is a well-defined map $\Phi: \Gamma_f \rightarrow \mathrm{Aut}(\mathbb{B}^N)$ 
with $\Phi(\gamma)=\psi$ where $(\gamma, \psi) \in A_f$. We next verify that $\Phi$ is a homomorphism.  Assume
$\Phi(\gamma_1)=\psi_1$ and $\Phi(\gamma_2)=\psi_2$. By the definition of $\Phi$, we
have $\psi_1 \circ f = f \circ \gamma_1$ and $\psi_2 \circ f =f \circ \gamma_2$.
Consequently,
$$\psi_2 \circ \psi_1 \circ f =\psi_2 \circ f  \circ \gamma_1=f \circ \gamma_2 \circ
\gamma_1.$$
We conclude that $\Phi(\gamma_2 \circ \gamma_1)=\psi_2
\circ \psi_1=\Phi(\gamma_2) \circ \Phi(\gamma_1).$ Hence $\Phi$ is a homomorphism.
The definition of equivariance now yields the equivariance of $f$ with respect to $\Phi$,
completing the proof.

\medskip

Let $f:\mathbb{B}^n \to \mathbb{B}^N$ be a minimal map. By Theorem 3.1, $f$
induces a homomorphism $\Phi: \Gamma_f \to \mathrm{Aut}(\mathbb{B}^N)$. We
say $\Phi$ {\it represents} $\Gamma=\Gamma_f$ in $\mathrm{Aut}(\mathbb{B}^N)$.

We discuss the kernel of the induced homomorphism $\Phi$ when $f$ is proper.
The behavior of $\Phi$ then depends on the
boundary regularity of $f$. Propositions 3.4 and 3.5 are known but
are expressed in different language in the literature. The proof of Proposition 3.2
uses the following simple fact.

\begin{remark} A proper holomorphic map is {\it finite}:
the inverse image of a point is a finite set. \end{remark}

\begin{proposition}
Let $f:\mathbb{B}^n \to \mathbb{B}^N$ be a minimal proper map. Let $\Phi$ be the induced homomorphism.
Then $\mathrm{Ker}(\Phi)$ a finite subgroup of
$\mathrm{Aut}(\mathbb{B}^n)$.
\end{proposition}

\begin{proof}
That $\mathrm{Ker}(\Phi)$ is a subgroup of
$\mathrm{Aut}(\mathbb{B}^n)$ is clear. First we show that $\mathrm{Ker}(\Phi)$ is closed in
$\mathrm{Aut}(\mathbb{B}^n).$
Assume $\gamma_\nu \in \mathrm{Ker}(\Phi) $ and $\gamma_\nu$ converges to $\gamma
\in \mathrm{Aut}(\mathbb{B}^n)$. By continuity, we conclude that $\gamma \in
\mathrm{Ker}(\Phi)$.

Next we claim $\mathrm{Ker}(\Phi)$ is compact. If not, 
then $\mathrm{Ker}(\Phi)$ moves some point $a$ in $\mathbb{B}^n$ arbitrarily close
to the boundary. By Remark 3.1,  $f^{-1}(f(a))$
is a finite set, yielding a contradiction.

Thus $\mathrm{Ker}(\Phi)$ is a compact Lie subgroup of $\mathrm{Aut}(\mathbb{B}^n).$
By standard Lie group theory (cf. [HT]), it is contained in a conjugate of ${\bf
U}(n).$

We finally claim that $\mathrm{Ker}(\Phi)$ is finite. It
is contained in a conjugate of ${\bf U}(n)$. If it were infinite, then there would exist $a \in
\mathbb{B}^n$ such that $\{ \gamma(a): \gamma \in \mathrm{Ker}(\Phi)\}$ is infinite. This conclusion 
contradicts $f$ being a finite map.
\end{proof}

\begin{proposition}
Let $G$ be a finite subgroup of $\mathrm{Aut}(\mathbb{B}^n)$. Then there exists an 
$N$ and a holomorphic proper map $f: \mathbb{B}^n \to \mathbb{B}^N$ such that $\mathrm{Ker} (\Phi)=G$.
\end{proposition}

\begin{proof}
We first assume $G$ is a subgroup of ${\bf U}(n).$ In this case, the statement was
proved in [Li1] (See Theorem 4.3.4 there). Next we assume $G$ is an arbitrary finite
subgroup of $\mathrm{Aut}(\mathbb{B}^n).$ By Lie group theory, $G$ is contained in a
conjugate 
of ${\bf U}(n).$ For some $\chi \in \mathrm{Aut}(\mathbb{B}^n),$ we thus have
$G_0=\chi \circ G \circ \chi^{-1} \subseteq {\bf U}(n).$ By the result above for ${\bf
U}(n),$ there is a holomorphic proper map $g$ for which $\mathrm{Ker} (\Phi_g)=G_0.$
Then $\mathrm{Ker} (\Phi_f)=G$ if $f=g \circ \chi.$ 

\end{proof}

We recall the notion of fixed-point-free subgroup of ${\bf U}(n)$. A finite subgroup $G$
of ${\bf U}(n)$ is {\it fixed-point-free} if the only element in $G$ with an eigenvalue
of $1$ is the identity. Thus the origin is the only fixed point in ${\bf C}^n$ under the action
of $G$. Equivalently, $G$ is fixed point free if each element of the group other than the identity
has no fixed points on the unit sphere. The next result 
appears in [Li1]. In the language of this paper, it shows that a boundary 
regularity assumption on  $f$ puts restrictions on
$\mathrm{Ker}(\Phi)$. We sketch the proof.

\begin{proposition}
Let $f: \mathbb{B}^n \rightarrow \mathbb{B}^N$ be a holomorphic proper map that is
continuously differentiable on the closed ball.
Then $\mathrm{Ker} (\Phi)$ is conjugate to a fixed-point-free finite subgroup
of ${\bf U}(n)$.
\end{proposition}

\begin{proof}
After composing $f$ with a source automorphism, we can assume that $\mathrm{Ker}(\Phi)$
is contained in ${\bf U}(n)$. Suppose $\gamma \in \mathrm{Ker} (\Phi)$ and that
$\gamma$ is not the identity. We show that $1$ cannot be an eigenvalue of $\gamma$.

Since $f$ is $C^1$-smooth to the boundary, by a classical Hopf lemma argument, $f$
is of full rank at each boundary point and thus is a local embedding there. Assume
for some $\gamma \in \mathrm{Ker}(\Phi)$ that $1$ is an eigenvalue of $\gamma$.
Let $E$ be the corresponding eigenspace. If $\gamma$ is not the identity,
then $E$ is a proper subspace. Then $\gamma$ maps any point $q$ close to
but not on $E$ to a different point nearby. Since $f \circ \gamma =f$, the map
$f$ cannot be injective near $E$. But $E$ intersects the sphere, contradicting
$f$ being a local embedding there.\end{proof}

When $f$ is $C^{\infty}-$smooth up to the boundary, $\mathrm{Ker} (\Phi)$
is quite restricted. The following result  gives a complete list of the possible
groups that can arise. See [D] for a proof
and for the corresponding maps. See also [G] for related results when the target 
is a generalized ball.

\begin{proposition}
Let $f:{\mathbb B}^n \to {\mathbb B}^N$ be a rational holomorphic proper map. Then $\mathrm{Ker} (\Phi)$ is cyclic.
Furthermore $\mathrm{Ker} (\Phi)$ is
conjugate to one of the following:
\begin{itemize}
\item $G$ is the cyclic group generated by $\eta {\bf I}_n$, where $\eta$ is a primitive $m$-th root of unity.
(Each $m$ is possible.)
\item $G$ is the cyclic group generated by $\eta {\bf I}_j \oplus \eta^2 {\bf I}_k$, 
where $\eta$ is a primitive odd root of unity. Here $j+k=n$. (Each odd natural number is possible.)
\item $G$ is the cyclic group of order $7$ generated by 
$\eta {\bf I}_j \oplus \eta^2 {\bf I}_k \oplus \eta^4{\bf I}_l$. Here $j+k+l=n$.

\end{itemize}
\end{proposition}

\begin{remark} The map $z \mapsto z^{\otimes m}$ provides an example of the first type. When $n=2$,
maps of the second type have the largest possible degree given the target dimension. See [DKR].
The simplest example of the third type maps ${\mathbb B}_3$ to ${\mathbb B}_{17}$ and appears in [D]. \end{remark}

When $f:{\mathbb B}^n \to {\mathbb B}^m$, one can artificially increase
the target dimension by considering the map $0 \oplus f$. This map is not minimal.
Lemma 3.2 shows how this construction impacts the group $H_f$. We close this section with a simple result showing
that this construction does not change $\Gamma_f$. 

\begin{proposition}
Let $f:{\mathbb B}^n \to {\mathbb B}^m$ be minimal. For $N>m$, put $g=0 \oplus f$. Then $\Gamma_f = \Gamma_g$.
\end{proposition}

\begin{proof}
Put $k=N-m$. We may assume that $f(0)=0$. If $(\gamma, \psi) \in A_f$, then
$$ 0 \oplus (f \circ \gamma) = \psi(0 \oplus f). \eqno (2.1) $$
Let $\phi \in {\rm Aut}({\mathbb B}^n)$ be such that $(\phi \circ f \circ \gamma)(0)=0$.
We may regard $({\bf I}_k, \phi)$ as an element $\varphi$ of ${\rm Aut}({\mathbb B}^N)$.
Then $\varphi( 0 \oplus (f \circ \gamma))= 0 \oplus (\phi \circ f \circ \gamma)$. Hence $\varphi$ maps $0$ to $0$.
By (2.1), 
$$ 0 \oplus (\phi \circ f \circ \gamma) = (\varphi \circ \psi) (0 \oplus f) = (\varphi \circ \psi)(g). \eqno (2.2) $$
Evaluating at $0$ shows that the automorphism $\varphi \circ \psi$ maps $0$ to $0$ 
and hence is unitary. Since $f$ is minimal, $\mathrm{Span}_{\mathbb{C}}\{ g \}=0 \oplus \mathbb{C}^m$.
As in the proof of Lemma 3.2, we conclude that $\varphi \circ \psi$
lies in ${\bf U}(k) \oplus {\bf I}_m$. Now (2.2) implies $\phi \circ f \circ \gamma =f$.
Thus $\gamma \in \Gamma_f$ if and only if $\gamma \in \Gamma_g$.
\end{proof}

\section{Groups and regularity}

We develop the tools to prove Theorem 4.1.
The next lemma illustrates how group invariance allows
one to prove that certain power series are in fact polynomials.
We establish Theorem 4.1 by generalizing the proof of this lemma.

\begin{lemma} Let $\Omega$ be an open ball about $0 \in {\mathbb C}^n$ and suppose
that $H:\Omega \to {\mathbb C}^N$ is holomorphic. Assume for each $\gamma \in {\bf U}(1) \oplus ...\oplus {\bf U}(1)$
that $||H \circ \gamma||^2 = ||H||^2$. Then there is a monomial map $G$ and a unitary $U$ such that $H=U\circ G$.
\end{lemma}
\begin{proof} Since $H$ is holomorphic and $\Omega$ is a ball about $0$, we can expand $H$ in a convergent power
series about $0$. Put
$$ H(z) = \sum_\alpha C_\alpha z^\alpha. $$
Then we have
$$ ||H(z)||^2 = \sum_{\alpha,\beta} \langle C_\alpha, C_\beta \rangle z^\alpha {\overline z}^\beta. $$
Let $\gamma$ be a diagonal unitary matrix with eigenvalues $e^{i \theta_j}$. Then the assumption 
$||H \circ \gamma||^2 = ||H||^2$ implies
$$ \sum_{\alpha,\beta} \langle C_\alpha, C_\beta \rangle  z^\alpha {\overline z}^\beta = \sum_{\alpha,\beta} \langle C_\alpha, C_\beta \rangle z^\alpha {\overline z}^\beta  e^{i \theta(\alpha-\beta)}. \eqno (3) $$
By assumption, equation (3) holds for all choices of the $\theta_j$. Equating Taylor coefficients then shows, whenever $\alpha \ne \beta$, that $ \langle C_\alpha, C_\beta \rangle = 0$. 
Since each coefficient $C_\alpha$ lies in the finite-dimensional space ${\mathbb C}^N$, they are linearly
dependent and orthogonal. Hence there are only finitely many non-zero coefficient vectors. Thus $H$ is a polynomial. Since the coefficient vectors are orthogonal,
there is a unitary map $U$ such that $H=U\circ G$, where $G$ is the monomial map whose components are the monomials
arising in $H$. \end{proof}

The idea of using circular symmetries to study power series appears already in [BM].
We next establish Theorem 4.1 from the introduction. 

\begin{definition} Let ${\mathcal C} = {\mathcal C}(n)$ denote the collection of continuously differentiable maps
$\gamma: (-\pi, \pi) \to {\bf U}(1) \oplus ... \oplus {\bf U}(1)$ such that $\gamma(0)= {\bf I}_n$.
\end{definition}

For $\gamma \in {\mathcal C}$, we note that $\gamma'(0)= iL$, where $i^2=-1$ and $L=L^*$ is real and diagonal.
(The Lie algebra of the unitary group is the skew-Hermitian matrices.)
The eigenvalues of $L$ are real. When they all have the same sign, we obtain
a criterion guaranteeing that a holomorphic map is in fact a polynomial.

\begin{proposition} Let $\Omega$ be a ball about $0$ in ${\mathbb C}^n$ and
suppose $f:\Omega \to {\mathbb C}^N$ is holomorphic.
For some $\gamma \in{\mathcal C}$, assume that $||f \circ \gamma||^2 = ||f||^2$.
Put $L =-i \gamma'(0)$. If all the eigenvalues of $L$ have the same sign, then $f$ is a polynomial.
More generally, if $L$ has $k$ eigenvalues of the same sign, then there is a $k$-dimensional vector subspace
$V$ such that the restriction of $f$ to $V$ is a polynomial.
\end{proposition} 

\begin{proof} 
Put $\theta(t) = (\theta_1(t),..., \theta_n(t))$, where 
$\gamma(t)$ is diagonal with eigenvalues $e^{i \theta_j(t)}$.
We expand $f$ in a (vector-valued) power series convergent in $\Omega$:
$$ f(z) = \sum_\alpha c_\alpha z^\alpha. $$
We are given that $||f(z)||^2 = ||f(\gamma(t)z)||^2$ for all $t \in (-\pi,\pi)$ and all $z \in \Omega$.
Equating Taylor coefficients yields, for each pair $\alpha, \beta$ of multi-indices, that
$$ \langle c_\alpha, c_\beta \rangle = e^{i \theta(t) \cdot (\alpha - \beta)}\langle c_\alpha, c_\beta \rangle $$
and hence we have
$$ 0 = \langle c_\alpha, c_\beta \rangle \ \left(1 - e^{i\theta(t) \cdot (\alpha - \beta)} \right). \eqno (4) $$
Write
${\bf m} = (m_1,...,m_n)$, where the $m_j$ are the eigenvalues of $L$. Differentiate (4) and evaluate at $t=0$
to obtain
$$ 0 = \langle c_\alpha, c_\beta \rangle \left( {\bf m} \cdot (\alpha-\beta)\right). \eqno (5) $$
We will show that $c_\alpha = 0$ for all but finitely many $\alpha$. 

Suppose first that the eigenvalues are $L$ have the same sign. After replacing $t$ by $-t$
we may assume they are all positive. 
Let $K_1= \min(m_j)$ and $K_2= \max(m_j)$. Thus $K_2\ge K_1 > 0$. Let $\alpha$ and $\beta$
be multi-indices for which $ {\bf m}\cdot (\alpha - \beta) = 0$;
then
$$ K_1 |\alpha| = K_1 \sum \alpha_j \le \sum m_j \alpha_j  = \sum m_j \beta_j \le K_2 \sum \beta_j = K_2 |\beta|. \eqno (6) $$

Let $W$ be the (finite-dimensional) span of the coefficients $c_\alpha$. Choose $\alpha_1,...,\alpha_\nu $ such that
the $c_{\alpha_j}$ span $W$. Choose a multi-index $\eta$ with  $|\eta| > {K_2 \over K_1} \ |\alpha_j|$ for
$1 \le j \le \nu$. By (6), ${\bf m} \cdot (\eta - \alpha_j) \ne 0$ for all $j$.
Since $\langle c_\alpha, c_\beta \rangle \ {\bf m}\cdot ({\alpha - \beta}) = 0$ for all $\alpha, \beta$, we conclude
that $\langle c_\eta, c_{\alpha_j} \rangle = 0$ for all $j$. Therefore $c_\eta = 0$ and hence
there are only finitely many non-vanishing coefficient vectors.
Thus $f$ is a polynomial.

Next suppose that $L$ has $k$ eigenvalues of the same sign. After renumbering the coordinates
and replacing $t$ by $-t$ if necessary, we may assume that these are positive and
correspond to the first $k$ coordinates. Setting the remainder of the variables equal to $0$
puts us in the situation above. The conclusion follows.
\end{proof}

\begin{remark}One can draw stronger conclusions. For example, when $n=1$,
equation (5) implies that the vectors $c_\alpha$ are mutually orthogonal
and hence that $f$ is an orthogonal sum of monomials. \end{remark}

\begin{remark} When $L$ has $N_+$ positive and $N_{-}$ negative eigenvalues
we can conclude (with obvious notation) that $f(z',0)$ and $f(0, z'')$ are both polynomials. 
When $L$ has eigenvalues of both signs, however, $f$ need not be a polynomial. \end{remark}

\begin{example} Let $\gamma(t)$ be the diagonal unitary matrix
with eigenvalues $e^{it}$ and $e^{-it}$. The eigenvalues of $L$ are then $\pm 1$.
Assume $f$ is a function of the product $z_1z_2$.
Then $f(\gamma(t)z) = f(z)$ but $f$ need not even be rational.
\end{example}

The following results are corollaries of Proposition 4.1.

\begin{theorem} Let $f: {\mathbb B}^n \to {\mathbb B}^N$ be holomorphic with $f(0)=0$.
If the Lie algebra $\mathfrak{g}$ of $\Gamma_f \cap {\bf U}(n)$
contains a matrix $M$ such that $L=-iM$ has 
$k$ positive eigenvalues, then there is a $k$-dimensional linear subspace $V$
such that the restriction of $f$ to $V$ is a polynomial.  \end{theorem}

\begin{proof} Write $M= iPDP^*$ where $P$ is unitary and $D$ is diagonal with $k$ positive eigenvalues.
Let $W$ be the linear subspace of $\mathfrak{g}$ spanned over ${\mathbb R}$ by $M$. Since $W$ is a Lie subalgebra,
there is a unique Lie subgroup $G \subseteq \Gamma_f \cap {\bf U}(n)$ whose Lie algebra is $W$.
We regard $G$ as the group of transformations of the form $Pe^{iDt}P^*$ for $t \in {\mathbb R}$. Put $h= f \circ P$. 
By Proposition 2.2, the curve $\gamma = e^{iDt}$ lies in $\Gamma_h$. We wish to apply 
Proposition 4.1. The condition on $\gamma'(0)$ holds by construction. Since $h(0)=0$,
Lemma 2.1 implies that $||h \circ \gamma||^2 = ||h||^2$. Thus both conditions
in Proposition 4.1 apply, and we conclude that the restriction of $f$ to some $k$-dimensional subspace
is a polynomial.
\end{proof}

The corollaries from the introduction follow easily. First,
if $\Gamma_f$ contains the center of ${\bf U}(n)$,
then the Lie algebra $\mathfrak{g}$ of $\Gamma_f \cap {\bf U}(n)$ contains $i{\bf I}_n$.
The subspace $W$ in Theorem 4.1 is ${\mathbb C}^n$ and hence $f$ is equivalent to a polynomial.
The statements in Corollary 1.2 follow as well. In each case the hypotheses
imply the hypotheses of Theorem 4.1 with $k=n$.  Hence $f$
is rational. The results from [DX] then yield the conclusions
about spherical equivalence.

Huang (See [Hu1]) established a linearity result for proper maps in low codimension
without assuming rationality. Instead, the maps are assumed to have two continuous
derivatives at the boundary sphere. By contrast, we make no regularity assumption;
instead we assume that the Hermitian invariant group is large.

We note also the following result from [DX]. Each finite subgroup of ${\rm Aut}({\mathbb B}^n)$
is the Hermitian group of some rational proper map with source ${\mathbb B}^n$.

\section{Completely non-rational proper maps}

A proper holomorphic map $f:{\mathbb B}^n \to {\mathbb B}^N$ is {\bf completely non-rational}
if, for each positive-dimensional affine linear subspace $V$ of ${\mathbb C}^n$, the restriction of $f$ to $V$
is not rational.

\begin{example} Completely non-rational proper maps between balls exist.
Begin with an arbitrary bounded holomorphic map $g$ that is completely non-rational but continuous
on the closed ball. After dividing by a constant, we may assume $1- ||g||^2 > 0$ on the unit sphere. 
By the work of L\o w [Lw], there is a holomorphic
map $h$ such that $1-||g||^2 = ||h||^2$ on the sphere.
Hence $g \oplus h$ maps the sphere to the sphere, and hence defines a proper
map between balls. It is completely non-rational because $g$ is.
One subtlety here involves the regularity of $h$. In general one cannot
conclude that $h$ is any smoother than continuous. \end{example}

We prove the following result.

\begin{theorem} Let $f:{\mathbb B}^n \to {\mathbb B}^N$ be a  minimal completely non-rational proper mapping.
Then $\Gamma_f$ and $T_f$ are either both finite or both noncompact. \end{theorem}

\begin{proof} By Proposition 2.3, it suffices to prove the conclusion for $\Gamma_f$.
Assume $f(0)=0$. Let $\mathfrak{g}$ be the Lie algebra of $\Gamma_f \cap {\bf U}(n)$.
If $\Gamma_f \cap {\bf U}(n)$ is not finite, then its dimension as a Lie group is at least $1$.
Then $\mathfrak{g}$ contains a non-zero skew-Hermitian matrix $M$. Put $L=-iM$.
After possibly replacing $M$ with $-M$ we may assume that $L$ has a positive eigenvalue.
By Theorem 4.1, $f$ is rational on some linear subspace of dimension $1$, contradicting
the assumption that $f$ is completely non-rational. Hence  $\Gamma_f \cap {\bf U}(n)$ is finite.

Next let $\gamma$ be an arbitrary 
element of ${\rm Aut}({\mathbb B}^n)$. Put $G = f \circ \gamma$. Proposition 2.2 implies
that $\Gamma_g = \gamma^{-1} \circ f \circ \gamma$. Thus both $\Gamma_g \cap {\bf U}(n)$ and $\Gamma_f \cap \gamma \circ {\bf U}(n) \circ \gamma^{-1}$ 
are finite. By Proposition 2.3, $\Gamma_f$ is a Lie subgroup. Hence it is either noncompact or 
contained in a maximal compact Lie subgroup of ${\rm Aut}({\mathbb B}^n)$. In the second case, 
for some automorphism $\phi$, we have
$$ \Gamma_f \subseteq \phi \circ {\bf U}(n) \circ \phi^{-1}. \eqno (7) $$
By the previous argument, the intersection of the group on the right-hand side of (7) with $\Gamma_f$
is $\Gamma_f$ and must be finite.
\end{proof}

We close the paper by mentioning a rigidity problem for holomorphic mappings between 
compact hyperbolic spaces raised by Siu in [S]. See also [Hu2] for a CR-geometric formulation.
Cao and Mok ([CM]) established the following result: Let $(X,g)$ be a compact $n$-dimensional complex hyperbolic space form 
and $(Y,h)$ an $m$-dimensional hyperbolic space form with $m \le 2n-1$. Then a holomorphic immersion $f:X \to Y$
is necessarily a totally geodesic isometric immersion. The conclusion is not known for larger target dimensions.
The hyperbolic space form $X$ is the quotient of the unit ball by a lattice.

Siu's question has the following formulation in the language of this paper. Let $f:{\mathbb B}^n \to {\mathbb B}^N$
be a proper holomorphic map, and suppose $\Gamma_f$ contains a (co-compact) lattice. Must $f$ be a totally geodesic embedding
in the Poincar\'{e} metric? The third part of Corollary 1.2 from the
introduction draws this conclusion when the group $\Gamma_f$ is non-compact and contains the center of ${\bf U}(n)$.

\section{bibliography}

\medskip

[BM] S. Bochner and W. Martin, Several Complex Variables. Princeton Mathematical Series, Vol. 10, Princeton University Press, Princeton, N. J., 1948. 

\medskip

[CM] H. Cao and N. Mok, Holomorphic immersions between compact hyperbolic space
forms, {\it Invent. Math.} 100 (1990) no. 1, 49-62.

\medskip

[D] J. P. D'Angelo,  Several Complex Variables and the Geometry of Real Hypersurfaces,
CRC Press, Boca Raton, Fla., 1992.

\medskip

[DX] J. P. D'Angelo and M. Xiao, Symmetries in CR complexity theory, {\it Adv. Math.} 313(2017), 590-627.

\medskip 

[F1] F. Forstneri\v c,
Extending proper holomorphic maps of positive codimension,
{\it Inventiones Math.}, 95(1989), 31-62.

\medskip

[F2] F. Forstneri\v c, Proper holomorphic maps from balls,
{\it Duke Math. J.} 53 (1986), no. 2, 427-441. 

\medskip

[G] D. Grundmeier, Signature pairs for group-invariant Hermitian polynomials,
{\it Internat. J. Math.} 22 (2011), no. 3, 311-343. 

\medskip

[HT] K. Hoffman and C. Terp, Compact subgroups of Lie groups and locally compact
groups, {\it Proc. A. M. S.}, Vol. 120, No. 2 (1994), 623-634.

\medskip

[HJY] X. Huang, S. Ji, and W. Yin, Recent progress on two
problems in several complex variables,  {\it Proceedings of the ICCM 2007},
International Press, Vol I, 563-575.

\medskip

[Hu1] X. Huang,  On a linearity problem for proper maps between balls in
complex spaces of different dimensions, {\it J. Diff. Geom.} 51 (1999), 
no 1, 13-33.

\medskip

[Hu2] X. Huang, On some problems in several complex variables and CR
geometry, {\it Proceedings of the ICCM}, S. T. Yau, Ed., AMS/IP Studies in Advanced
Mathematics 20 (2001), 383-396.

\medskip

[Li1] D. Lichtblau, Invariant proper holomorphic maps between balls, Thesis,
University of Illinois at Urbana-Champaign, 1991.

(http://www.ideals.illinois.
edu/handle/2142/23603)

\medskip

[Li2] D. Lichtblau, Invariant proper holomorphic maps between balls, {\it Indiana Univ.
Math. J.} 41 (1992), no. 1, 213-231.

\medskip

[Lw] E. L\o w, Embeddings and proper holomorphic maps of strictly pseudoconvex domains into polydiscs and balls,
{\it Math. Z.} 190 (1985), no. 3, 401-410. 

\medskip

[KM]  V. Koziarz and N. Mok, Nonexistence of holomorphic submersion between complex
unit balls equivariant with respect to a lattice and their generalization, {\it Amer. J.
Math.} 132 (2010)  no. 5, 1347-1363.

\medskip

[S] Y.-T. Siu, Some recent results in complex manifold theory related to vanishing
theorems for the semipositive case, Arbeitstagung Bonn 1984, Lect. Notes Math. Vol.
1111, Springer-Verlag 1985, Berlin-Heidelberg-New York, pp. 169-192.

\end{document}